\documentclass[a4paper,12pt]{amsart}
\usepackage[utf8]{inputenc}

\usepackage{amsfonts, amsmath, amssymb, amsthm,mathrsfs}
\usepackage{anysize}
\marginsize{2cm}{2cm}{2cm}{2cm}

\usepackage{graphicx}
\usepackage{hyperref}
\hypersetup{colorlinks,%
citecolor=green,%
linkcolor=cyan,%
pdftex}
\usepackage{color}

\DeclareMathOperator{\pos}{pos}
\DeclareMathOperator{\tcone}{tcone}
\DeclareMathOperator{\simplex}{simplex}

\makeatletter
\newtheorem*{rep@theorem}{\rep@title}
\newcommand{\newreptheorem}[2]{%
\newenvironment{rep#1}[1]{%
 \def\rep@title{#2 \ref{##1}'}%
 \begin{rep@theorem}}%
 {\end{rep@theorem}}}
\makeatother

\theoremstyle{plain}
\newtheorem{theorem}{Theorem}
\newtheorem{lemma}{Lemma}
\newtheorem{corollary}[lemma]{Corollary}

\newtheorem{claim}{Claim}
\newreptheorem{theorem}{Theorem}

\theoremstyle{definition}
\makeatletter
\newtheorem*{rep@definition}{\rep@title}
\newcommand{\newrepdefinition}[2]{%
\newenvironment{rep#1}[1]{%
 \def\rep@title{#2 \ref{##1}'}%
 \begin{rep@definition}}%
 {\end{rep@definition}}}
\makeatother

\newtheorem{definition}{Definition}
\newrepdefinition{definition}{Definition}

\theoremstyle{remark}
\newtheorem{remark}[lemma]{Remark}

\newcounter{fig}
\renewcommand{\figure}{\refstepcounter{fig}%
 Fig. \arabic{fig}. }

\title{Algebraic vertices of non-convex polyhedra}

\author{Arseniy~Akopyan}
\address{Arseniy~Akopyan, Institute of Science and Technology Austria (IST Austria), Am Campus~1, 3400 Klosterneuburg, Austria}
\email{akopjan@gmail.com}

\author{Imre B\'ar\'any} 
\address{Imre B\'ar\'any, MTA Alfr\'ed R\'enyi Institute of Mathematics,
PO Box 127, H-1364 Budapest, Hungary
and Department of Mathematics, University College London,
Gower Street, London, WC1E 6BT, U.K.}
\email{barany@renyi.hu}

\author{Sinai Robins}
\address{Sinai Robins, Institute of Mathematics and Statistics, University of S\~ao Paulo, Brasil}
\email{sinai.robins@gmail.com}


\begin{document}

\begin{abstract}
	In this article we define an \textbf{algebraic vertex} of a generalized polyhedron and show that the set of algebraic vertices is the smallest set of points needed to define the polyhedron.  We prove that the indicator function of a generalized polytope $P$ is a linear combination of indicator functions of simplices whose vertices are algebraic vertices of $P$.
	We also show that the indicator function of any generalized polyhedron is a linear combination, with integer coefficients, of indicator functions of cones with apices at algebraic vertices and line-cones.
	
	The concept of an algebraic vertex is closely related to the Fourier--Laplace transform.
	We show that a point $\mathbf{v}$ is an algebraic vertex of a generalized polyhedron $P$ if and only if the tangent cone of $P$, at $\mathbf{v}$,  has non-zero Fourier--Laplace transform.
	
\bigskip
\noindent

\end{abstract}

\maketitle

\section{Introduction}

We study the vertices of \textbf{non-convex polyhedra}, which we also call \textbf{generalized polyhedra}, and which we define as the finite union of convex polyhedra in $\mathbf R^d$.




There are many different ways to define a vertex of a generalized polyhedron $P$, most of them based on properties of the tangent cone to $P$ at a point $\mathbf{v} \in P$.
The tangent cone at $\mathbf{v}$, which we write as $\tcone(P, \mathbf{v})$, is intuitively the collection of all directions that we can `see' if we stand at $\mathbf{v}$ and look into $P$ (see Section~\ref{sec:algebra of polyhedra} for a rigorous definition).
We furthermore define a {\bf line-cone} to be a cone that is the union of parallel lines.

One approach is to say that a point ${\mathbf{v}}$ is a vertex of a generalized polyhedron $P$ if its tangent cone is not a {\bf line-cone}.
We call such a point $\mathbf{v}$ a {\bf geometric vertex} of a generalized polyhedron, and it is clear that for a {\em convex} polyhedron this definition coincides with the usual definition of vertices (see the last Section).
In this article we focus on another definition of vertices.

\begin{definition} \label{def:indicator definition of Fourier vertex through cones}
	For a generalized polyhedron $P$, a point $\mathbf{v} \in P$ is called an {\bf algebraic vertex} of $P$ if the indicator function of its tangent cone $\tcone(P, \mathbf{v})$ cannot be represented (up to a set of measure zero) as a linear combination of indicator functions of line-cones.
\end{definition}

The theorem of D.~Frettl\"{o}h and A.~Glazyrin \cite{frettloh2009lonely} states that the indicator function of a convex cone which is not a line-cone cannot be represented as a sum of indicator functions of line-cones, implying that the vertices of an ordinary convex polytope are indeed algebraic vertices.


Our main result is the following description of algebraic vertices, showing that in some sense these generalized vertices form a minimal set of points needed to describe a \textbf{generalized polytope}, which is by definition a bounded generalized polyhedron.

Throughout, we denote the \textbf{indicator function} of any set $S \subset \mathbf R^d$ by $[S]$.
 In other words $[S](x)=1$ if $x \in S$, and $[S](x)=0$ if $x \notin S$.

%
%

\begin{theorem}
	\label{thm:sum of simplices}
	Let $\mathcal{V}_P$ be the set of algebraic vertices of a generalized polytope $P \subset \mathbf R^d$, and let~$\mathcal{T}_P$ be the set of simplices whose vertices lie in $\mathcal{V}_P$.
	Then
	\begin{equation}
		\label{eq:polytope sum of simplices}
		[P]=\sum_{T_i \in \mathcal{T}_P} \alpha_i [T_i],
	\end{equation}
	where the $\alpha_i$ are integers and the equality holds throughout $\mathbf R^d$, except perhaps for a set of measure zero.

	Moreover, if $[P]$ is represented (up to measure zero) as a linear combination of indicator functions of some finite number of simplices, then the set of vertices of these simplices must contain~$\mathcal{V}_P$.
\end{theorem}

It seems that weaker versions of this theorem were known for a long time, and in particular A.~Gaifullin \cite{gaifullin2014generalization} showed that the indicator function of a polytope $P$ is a linear combination of indicator functions of simplices whose vertices belong to the set of geometric vertices of the polytope $P$.

Algebraic vertices of $P$ are closely related to the structure of integral transforms of $[P]$, in particular the Fourier--Laplace transform of $[P]$ (see below), and the  Fantappi\`e transform $\mathcal{F}_P(z)$  \cite{gravin2012moments}.
In \cite{gravin2012moments} the authors prove  an analogue of Theorem~\ref{thm:sum of simplices}, 
where the role of algebraic vertices is played by linear functions $\langle z,v\rangle$, 
appearing in the denominator $\prod_{v\in V_{\mathcal{F}}}(1-\langle z,v\rangle)$ of $\mathcal{F}_P(z)$.
They needed, however, an extra genericity assumption on $V_{\mathcal{F}}$, and were only able to 
establish a decomposition like the one in \eqref{eq:polytope sum of simplices} with $\alpha_i\in\mathbf{R}$.
They stated a conjecture \cite[Conjecture~7]{gravin2012moments}, that claims the genericity is not needed; 
this conjecture is essentially settled in the affimative by our Theorem~\ref{thm:sum of simplices},
see Remark~\ref{rem:fanappie vertices} for further discussion.

We were informed by Dmitrii Pasechnik [private communication, 2014] that
\cite[Conjecture~7]{gravin2012moments} follows from results in \cite{brion1999arrangement}.

Generalized (non-bounded) polyhedra also can be described through their algebraic vertices with the following theorem.
\begin{theorem}
	\label{thm:sum of line-cones and simplicial cones}
	Let $\mathcal{V}_P$ be the set of algebraic vertices of a generalized polyhedron $P \subset \mathbf R^d$.
Then
\[
[P]= \sum_{i=1}^k \alpha_i [D_i] +  \sum_{\mathbf{v} \in \mathcal{V}_P} [\tcone(P,\mathbf{v})],
\]
for some integers $\alpha_i$ and line-cones $D_i$, $i=1, \dots, k$. The equality holds almost everywhere, except perhaps on a set of measure zero.

Moreover, if $[P]$ is represented (up to a set of measure zero) as a linear combination of indicator functions of line-cones and simplicial cones, then the set of apices of these simplicial cones should contain~$\mathcal{V}_P$.
\end{theorem}

It is well known that the indicator function of a line-cone has vanishing Fourier--Laplace transform (see the definition in section~\ref{sec:fourier transform}).
Therefore if $\mathbf{v}$ is not an algebraic vertex, then  the Fourier--Laplace transform of the indicator function of its tangent cone also vanishes, because it is a finite linear combination of indicator functions of line-cones.

We show that the opposite also holds.  Lemma~\ref{lem:polyconical zero-f is a zum of line-cones} implies that the indicator function of the tangent cone at an algebraic vertex has non-vanishing Fourier--Laplace transform.
We formulate this fact in a more general form.
\begin{theorem} \label{thm:f-zero cone is sum of line-cones}
	If $P$ is a generalized polyhedron with zero Fourier--Laplace transform, then it does not have algebraic vertices.
\end{theorem}

%
%
%
%

We prove these theorems in a more general form, for elements of the algebra of polyhedra instead of non-convex polyhedra.   
The article is organized in the following way. In section~\ref{sec:algebra of polyhedra} we define the algebra of polyhedra and tangent cones of generalized polyhedra.
In section~\ref{sec:proof of theorem of sum of simpices} we prove Theorems~\ref{thm:sum of simplices} and  \ref{thm:sum of line-cones and simplicial cones}.
In section~\ref{sec:fourier transform} we recall the definition of the Fourier--Laplace transform.
In section~\ref{sec:proof of the cone-like sum} we give the proof of Theorem~\ref{thm:f-zero cone is sum of line-cones}.
In the ensuing section we discuss one corollary of the proof, which states that
polytopes with zero Fourier--Laplace transform also have a zero Fourier--Laplace transform for each of their \emph{signed sections}.
Finally, in the last section, we discuss various  different ways of defining vertices of generalized polyhedra and find various relations between these definitions.

\bigskip
\section{The algebra of polyhedral indicator functions}
\label{sec:algebra of polyhedra}



We define a {\bf polyhedron} as the intersection of a finite number of half-spaces, which is a convex set. Thus, a polyhedron may be unbounded, and a bounded polyhedron is by definition a convex polytope.   We define a {\bf generalized polyhedron} as a finite union of polyhedra. Now we can speak about the union, intersection, and subtraction of polyhedra.

For each generalized polyhedron $P$ we can associate the class of sets in $\mathbf R^d$ which differ from~$P$ on a set of measure zero. Thus, we work with equivalence classes of generalized polyhedra. The above operations extend naturally to these equivalence classes.

\begin{lemma}
	\label{lem:union of simplicas}
	Each generalized polythedron $P \subset \mathbf R^d$ can be represented as a finite disjoint union of convex polyhedra.
\end{lemma}
\begin{proof}
	Draw a hyperplane through each facet of each convex polyhedron in the union generating~$P$.
	These hyperplanes cut $\mathbf R^d$ into convex cells (some of which are unbounded) that are by definition convex polyhedra. Some of them are contained in $P$, and their union is exactly~$P$.
\end{proof}

\noindent
We call such a representation of $P$ into a disjoint union of convex polyhedra a {\bf convex partition of $P$}.

\begin{lemma}
	\label{lem:p setminus q}
	If $P$ and $Q$ are two generalized polyhedra then the subtraction $T=P\setminus Q$ is also a generalized polyhedron.
\end{lemma}
\begin{proof} This argument is almost identical with the previous proof.
	Draw a hyperplane through each facet of each convex polyhedron of the convex partitions of $P$ and $Q$.
	These hyperplanes cut $\mathbf R^d$ into convex cells (some of which are unbounded). Some of these cells are contained in~$T$ and their union is $T$.
\end{proof}

There is a natural algebra structure on the equivalence classes of indicator functions of polyhedra, as follows.

Let $\mathcal{P}^d$ be the vector space, under addition of functions, of real linear combinations of indicator functions of all convex polyhedra in $\mathbf R^d$. We may also define multiplication as the pointwise multiplication of functions. These two operations allow us to consider $\mathcal{P}^d$ as
 an algebra, which we call the {\bf algebra of polyhedra}.

 Again, we are working with equivalent classes of functions. So two functions are equivalent if they differ on a set of measure zero. In fact, $\mathcal{P}^d$ is the algebra of equivalence classes of polyhedra. 
Furthermore, if $P$ and $Q$ are two generalized polyhedra then:
\begin{gather*}
	[P \cap Q] = [P]\cdot [Q],\\
	[P \cup Q] = [P]+[Q]-[P]\cdot [Q].
\end{gather*}

From the previous two lemmas it follows that for any function $f$ belonging to the algebra~$\mathcal{P}^d$, the level set $\{\mathbf{x} \in \mathbf R^d \mid f(\mathbf{x}) = {\text constant} \}$ is a generalized polyhedron, because it can be represented as union and subtraction of convex polyhedra in $\mathbf R^d$.
So we can represent $f$ as a sum $\sum \alpha_i [P_i]$, where the $P_i$ are \emph{disjoint} generalized polyhedra, and $\alpha_i \in \mathbf{R}$.
We call this sum a {\bf basic decomposition} of $f$.

If $P$ is a convex polyhedron and $\mathbf{v}$ is a point then we define the  tangent cone of $P$ at $\mathbf{v}$ by
\begin{equation*}
	\begin{array}{l}
				\tcone(P, \mathbf{v})=\{\mathbf{v}+\mathbf{x}: \mathbf{v}+\varepsilon \mathbf{x} \in P \text{ for some } \epsilon>0\}, \text{ if } \mathbf{v} \in P, \\
				\tcone(P, \mathbf{v})=\emptyset, \text{ if } \mathbf{v} \notin P.
	\end{array}
\end{equation*}

If the generalized polyhedron $P$ is a union of polyhedra $P_1$, $\dots$, $P_n$ then we define $\tcone(P, \mathbf{v})$ by their union $\cup_{i=1}^n \tcone(P_i, \mathbf{v})$.

Let $\sum \alpha_i [P_i]$ be a basic decomposition of $f$, and let
 $C_i$ be the tangent cone at a point $\mathbf{v}$ belonging to the polyhedron $P_i$.
We define the \textbf{tangent cone} of the function $f$, at the point~$\mathbf{v}$, by
 $\tcone(f,\mathbf{v}) := \sum \alpha_i [C_i]$.

As an aside, an alternate way to define the tangent cone to $f$ is to consider a sufficiently small neighborhood $U_{\mathbf{v}}$ of ${\mathbf{v}}$ and for each point $\mathbf{x} \in U_{\mathbf{v}}$ define $\tcone(f,\mathbf{v})(\mathbf{y})=f(\mathbf{x})$ for all
points $\mathbf{y}$ on the ray $[{\mathbf{v}}, \mathbf{x})$.

We are now ready to define the algebraic vertices of a function $f \in \mathcal{P}^d$. We say that $\mathbf{v}$ is an \textbf{algebraic vertex} of $f$ if $\tcone(f,\mathbf{v})$ cannot be represented as $\sum \alpha_i [D_i]$, where the $D_i$ are line-cones (cones that can be represented as a union of parallel lines) and $\alpha_i \in \mathbf{R}$.

We say that $f$ is {\bf polyconical function} if it can be represented as a linear combination of indicator functions of simplicial cones $C_i$ with apex at the origin:
	\begin{equation*}
		f=\sum_{i=1}^n \alpha_i [C_i].
	\end{equation*}

The collection of {\bf polyconical functions} forms a very natural subalgebra of $\mathcal{P}^d$. 
We will also need the following definition.

\begin{definition}
\label{def:section by hyperplane}
		Let $f$ be a function on $\mathbf{R}^d$, and let $h$ be an oriented hyperplane, which means that there is a normal vector $\mathbf{n}_h$, whose direction depends on the orientation of $h$.
		Suppose that for almost all (up to measure zero) $\mathbf{x} \in h$ the limits $f_h^+(\mathbf{x})=\lim\limits_{\varepsilon \rightarrow 0} f(\mathbf{x} + \varepsilon \mathbf{n}_h)$ and $f_h^-(\mathbf{x})=\lim\limits_{\varepsilon \rightarrow 0} f(\mathbf{x}-\varepsilon \mathbf{n}_h)$  exist.
		Then the function $f_h(\mathbf{x}):=f_h^+(\mathbf{x})-f_h^-(\mathbf{x})$ is called the \bf{signed section of the function $f$ by the signed hyperplane~$h$}.
\end{definition}

It is clear that for $f\in \mathcal{P}^d$, both $f_h^+(\mathbf{x})$ and $f_h^-(\mathbf{x})$ belong to $\mathcal{P}^{d-1}$ in the $(d-1)$-dimensional hyperplane $h$.
Therefore the signed section of $f\in \mathcal{P}^d$ is also an element of the algebra of polyhedra $\mathcal{P}^{d-1}$ defined on $h$.

In what follows we prove Theorems~\ref{thm:sum of simplices}, \ref{thm:sum of line-cones and simplicial cones}, \ref{thm:f-zero cone is sum of line-cones} not only for generalized polyhedra and polytopes, but also for elements of $\mathcal{P}^d$.



%

\section{Proofs of Theorems~\ref{thm:sum of simplices} and \ref{thm:sum of line-cones and simplicial cones}}
\label{sec:proof of theorem of sum of simpices}
Here we prove the statement of the Theorems for functions with bounded support from the algebra $\mathcal{P}^d$.
\begin{reptheorem}{thm:sum of simplices}
	Assume $f \in \mathcal{P}^d$ has bounded support and let $\mathcal{V}_f$ be the set of algebraic vertices of $f$. Let $\mathcal{T}_f$ be the set of simplices whose vertices lie in $\mathcal{V}_f$.
	Then for suitable real $\alpha_i$,
	\begin{equation*}
		\label{eq:polytope sum of simplices gen}
		f=\sum_{T_i \in \mathcal{T}_P} \alpha_i [T_i].
	\end{equation*}
	
	Moreover, if $f$ is represented as a linear combination of indicator functions of simplices, then the set of vertices of these simplices must contain $\mathcal{V}_f$.
	
	If the coefficients of $f$ in its basic decomposition are all integers, then  the $\alpha_i$ are also  integers.
\end{reptheorem}
	
	\bigskip
Consider any oriented hyperplane $h$ and the signed section $f_h$ of $f$ by $h$.
Usually $f_h$ is zero, but sometimes it is not.
If $P$ is convex polytope then $[P]_h$ is non-zero if and only if $h$ contains a facet of $P$.
Using this observation we can define a facet of $f$ in the following way.
\begin{definition}
	\label{def:facet of polytope}
	If $f \in \mathcal{P}^d$,  and the signed section $f_h$ is non-zero for an oriented hyperplane $h$, then the function $f_h$ is called a \textbf{facet} of $f$.
\end{definition}

Note that if $f_h$ is a facet of $f$ with the same orientation as $h$, then it is also a facet with the opposite orientation of $h$.
It is clear that $f_h$ is a function belonging to $\mathcal{P}^{d-1}$, on the hyperplane~$h$.
We will use the following Lemma.

\begin{lemma}
	\label{lem:vertex of a facet is a vertex of polytope}
	An algebraic vertex of any facet of the function $f$ is an algebraic vertex of~$f$.
\end{lemma}

This result follows directly from \textbf{Lemma~\ref{lem:section of polyconical functions}} in Section~\ref{sec:signed section}.
We offer a more direct proof here, though, of Lemma~\ref{lem:vertex of a facet is a vertex of polytope} using Definition~\ref{def:indicator definition of Fourier vertex through cones} only.

\begin{proof}
	We show that if a point $\mathbf{v}$ is not an algebraic vertex of $f$ then it is not an algebraic vertex of any signed plane $h$ containing it.
	Hence we suppose that $\tcone(f,\mathbf{v}) = \sum \alpha_i [D_i]$, where $D_i$ are line-cones.
	Note that $[D_i]_h$ are also indicator functions of some line-cone in $h$.
	Since taking signed section is a linear map from $\mathcal{P}^d$ to $\mathcal{P}^{d-1}$, it follows that $\tcone(f_h,\mathbf{v}) = \sum \alpha_i [D_i]_h$.
	Therefore $\mathbf{v}$ is not an algebraic vertex of $f_h$.
\end{proof}

Lemma \ref{lem:vertex of a facet is a vertex of polytope} implies, by an easy induction, the following corollary.
\begin{corollary}
	\label{cor:fourier-vertex exists in polytopes}
	Any non-zero $f \in \mathcal{P}^d$ with bounded support has an algebraic vertex.
\end{corollary}

\begin{proof}[Proof of Theorem~\ref{thm:sum of simplices}']
The proof is by induction on the dimension. For the $1$-dimensional space the statement is trivial.

Let's fix a point $\mathbf{p}$ in $\mathbf R^d$.
For each facet $f_h$ of $f$ choose the orientation of $h$ so that $p$ lies in the positive halfspace of $h$.
We can associate a signed ``pyramid function'' $Pyr_h$ as follows:
if $\mathbf{x} \in h$ and $\mathbf{y}$ is a point on the interval $(\mathbf{p}, \mathbf{x})$ define $Pyr_h(\mathbf{y})=f_h(\mathbf{x})$ and zero otherwise.
(So, if $f_h$ is an indicator function of a two-dimensional polygon in $\mathbf{R}^3$ then $Pyr_h$ will be the indicator function of the pyramid whose apex is $\mathbf{p}$ and whose base is the polygon.)

It is clear that
\begin{equation}
	\label{eq:polytope as a sum of pyramid}
	f=\sum_{h} Pyr_h,
\end{equation}
where the sum is taken over all facets of $f$.
Indeed for each point $\mathbf{y}$ consider the tail $[\mathbf{y}, \infty)$ of the ray $[\mathbf{p}, \mathbf{y})$, intersecting hyperplanes the $h_1$, $h_2$, \dots, $h_k$ that are supporting facets of $f$.
As their orientation was chosen properly, $y$ lies on the positive side of each $h_i$, and $f$ equals $0$ at points that are
far away on the ray, we see that $f(\mathbf{y})$ equals the sum of $f_{h_i}$.

By the induction hypothesis each facet $f_h$ is a linear combination of indicator functions of simplices:
	\begin{equation*}
		f_h=\sum \beta_{h,i} [T_{h,i}],
	\end{equation*}
	where $T_{h,i}$ are simplices whose vertices are algebraic vertices of $f_h$ (which are algebraic vertices of $f$ by Lemma~\ref{lem:vertex of a facet is a vertex of polytope}).
Note that
	\begin{equation}
		\label{eq:representation of pyramid}
		Pyr_h=\sum \beta_{h,i} [\simplex(\mathbf{p}, T_{h,i})],
	\end{equation}
	where the $\simplex(\mathbf{p}, T_{h,i})$ is a simplex with base $T_{h,i}$ and apex $\mathbf{p}$, and  $\mathbf{p}$ is an arbitrary point of $\mathbf{R}^d$.

	Now we let $\mathbf{p}$ be any algebraic vertex of $f$, which exists by Corollary~\ref{cor:fourier-vertex exists in polytopes}. We conclude that the vertices of all simplices $Pyr(\mathbf{p}, T_{h,i})$ are algebraic vertices of $f$.

	Therefore combining equalities \eqref{eq:polytope as a sum of pyramid} and \eqref{eq:representation of pyramid} we obtain the required representation:

	\begin{equation*}
		\label{eq:polytope as a sum of certain pyramid}
		f=\sum_{h} \sum \beta_{h,i} [\simplex(\mathbf{p}, T_{h,i})].
	\end{equation*}
	
	It is clear that if the coefficients in the basic decomposition of $f$ are integers, then the coefficients of the basic decomposition of its signed sections are also integers.
	Therefore in the representation of $f$ which we obtain by this induction the coefficients are integers.

	Finally we show that if $f=\sum_1^m \alpha_i [T_i]$, where each $T_i$ is a simplex, then the set of vertices of $T_i$ contains $\mathcal{V}_f$.
	Suppose $\mathbf{v} \in \mathcal{V}_f$ is not a vertex of any of the simplices $T_i$.
	Then the tangent cone $\tcone([T_i], \mathbf{v})$ is a line-cone for every $i$.
	Since $\tcone (f,\mathbf{v})=\sum_i \alpha_i \tcone([T_i],\mathbf{v})$, the tangent cone of  $f$ at $\mathbf{v}$ can be represented as a sum of line-cones. This contradicts the fact that $\mathbf{v}$ is an algebraic vertex of $f$.
\end{proof}

\begin{reptheorem}{thm:sum of line-cones and simplicial cones}
Let $\mathcal{V}_f$ be the set of algebraic vertices of $f \in \mathcal{P}^d$.
Then
\[
f= \sum_{i=1}^k \alpha_i [D_i] +  \sum_{\mathbf{v} \in \mathcal{V}_f} \tcone(f,\mathbf{v}),
\]
for some reals numbers $\alpha_i$ and line-cones $D_i$, $i=1, \dots, k$.

Moreover, if $f$ is represented as a linear combination of indicator functions of line-cones and simplicial cones, then the set of apices of these simplicial cones should contain~$\mathcal{V}_f$.

	If the coefficients of $f$ in its basic decomposition are all integers, then  $\alpha_i$ and $\beta_i$ are also integers.
\end{reptheorem}

\begin{proof}
	Draw a hyperplane through each facet of each polyhedron in the basic decomposition of~$f$.
	These hyperplanes separate the space into convex polyhedra $P_i$, which tile the space face-to-face.
	Denote the set of vertices of this convex partition by $\mathcal{V}$.
	On each of these polyhedra $P_i$ the function $f$ takes the same constant value, so:
	\[
	f=\sum_i \gamma_i [P_i].	
	\]
	
	We note that each convex polyhedron can be represented as a linear combination of tangent cones of its faces of all dimensions, also known as the Brianchon--Gram relation \cite[Theorem 9.5]{beck2007computing}. Since the tangent cone of a face of dimension $\geq 1$ is a
	 line-cone, we may write:
	\[
	[P_i]=\sum_{{\mathbf{v}} \in \mathcal{V}} \tcone(P_i, {\mathbf{v}}) + \sum_j \alpha_{i,j} [D_j],
	\]
	where $D_j$ are line-cones which are tangent cones corresponding to faces of $P_i$ of dimension $\geq 1$.
	
	Observe now that if $\mathbf{v}$ is not an algebraic vertex of $f$ then $\sum_i\gamma_i\tcone(P_i, {\mathbf{v}})=\tcone(f, {\mathbf{v}})$ can be represented as a linear combination of indicator functions of line-cones.
	 We see that $f$ is a linear combination of indicator functions of line-cones, plus a linear combination of tangent cones at algebraic vertices of $f$.
	 This proves the first part of the Theorem.
	
	It is clear that if the coefficients of $f$ in its basic decomposition are all integers, then in each step we obtain a representation of $f$ with integers coefficients.
	
	Finally, we show that if $f$ is represented as a linear combination of indicator functions of line-cones and simplicial cones, then the set of apices of these simplicial cones should contain~$\mathcal{V}_f$.
	This is similar to the proof of the corresponding statement in Theorem~\ref{thm:sum of simplices}'.
	Suppose that
 \[
 f=\sum_i \alpha_i [D_i]+\sum_i \beta_i [C_i],
 \]
 	for some real $\alpha_i$ and $\beta_i$, line-cones $D_i$, $i=1, \dots, k$, and simplicial cones $C_i$, $i=1, \dots, m$.
	Let $\mathbf{v}\in \mathcal{V}_f$. Note that
	 \[
	 \tcone (f,\mathbf{v})=\sum_i \alpha_i \tcone([D_i],\mathbf{v})+\sum_i \beta_i \tcone([C_i],\mathbf{v}). \]
	If $\mathbf{v}$ is not a vertex of some cone $C_i$, then all $\tcone([C_i], \mathbf{v})$ are line-cones, and $\tcone([D_i],\mathbf{v})$ are line-cones as well.
	Therefore we obtain a representation of $\tcone (f,\mathbf{v})$ as a linear combination of indicator functions of line-cones, contradicting the assumption that $\mathbf{v}$ is an algebraic vertex.
\end{proof}

\medskip
\section{The Fourier--Laplace transform}
\label{sec:fourier transform}
In this section we recall the definition of the Fourier--Laplace transform and list some of its properties. For more details see \cite[Chapter 8]{barvinok2008integer} and \cite[Chapter 10]{beck2007computing}, or the original works \cite{lawrence1991polytope} and \cite{pukhlikov1992riemann}.

The Fourier--Laplace transform of a generalized $d$-polytope $P$ is defined as follows:
\begin{equation*}
\widehat{[P]}(\mathbf{z}) := \int_{P} e^{\langle \mathbf{z}, \mathbf{x} \rangle} d\mathbf{x},
\end{equation*}
valid for all $\mathbf{z} \in \mathbf{C}^d$.

It may be tempting to also define the Fourier--Laplace transform of generalized polyhedra $P$ in the same manner,
but we cannot do so because we run into the problem that there may not be any value of $\mathbf{z} \in \mathbf{C}^d$ for which the latter integral converges, as happens for example when $P$ consists of positive and negative orthants.

We first define the Fourier--Laplace transform of a {\bf convex polyhedral cone} $K$ whose apex is the origin by
\begin{equation}
	\label{eq:FL-integral cone}
\widehat{[K]}(\mathbf{z}) := \int_{K} e^{\langle \mathbf{z}, \mathbf{x} \rangle} d \mathbf{x},
\end{equation}
which converges for all $\mathbf{z} \in \mathbf{C}^d$, for which the real part $\Re(\mathbf{-z})$ lies in the interior of the \textbf{dual cone} $K^\circ$, that is $K^\circ=\{\mathbf{y} \in \mathbf{R}^d: \langle \mathbf{x}, \mathbf{y}\rangle \geq 0 , \forall \mathbf{x} \in K \}$.
 We recall that a
{\bf simplicial cone} $K$ is a convex cone in $\mathbf R^d$ with exactly $d$ edges (also called generators), say $\mathbf{w}_1, \dots, \mathbf{w}_d$, and for such a simplicial cone with apex at the origin, we have the elementary fact that  (see \cite{barvinok2008integer,beck2007computing}):
\begin{equation}\label{eq:transform of a simplicial cone}
\widehat{ [K]}(\mathbf{z}) :=
\frac{|\det K|}{\langle \mathbf{w}_1, \mathbf{z} \rangle \cdots \langle \mathbf{w}_d, \mathbf{z} \rangle}=
\frac{| \mathbf{w}_1\wedge \dots \wedge \mathbf{w}_{d}|}{\langle \mathbf{w}_1, \mathbf{z} \rangle \cdots \langle \mathbf{w}_d, \mathbf{z} \rangle},
\end{equation}
where we define $|\det K|$ to be the absolute value of the determinant of the matrix whose columns are the vectors $\mathbf{w}_j \in \mathbf R^d$.

Let $K_{\mathbf{v}}$ be a cone $K$ shifted by a vector $\mathbf{v}$.
From Equation~\eqref{eq:FL-integral cone} it follows that $\widehat {[K_{\mathbf{v}}]} (\mathbf{z})$ must be defined as $\widehat {[K]}(\mathbf{z}) \cdot e^{\langle \mathbf{z}, \mathbf{v} \rangle}$.

Note that if the edges $\mathbf{w}_1$ and $\mathbf{w}_2$ of a simplicial cone $K$ are parallel to each other then it degenerates to a line-cone and the value in the numerator in \eqref{eq:transform of a simplicial cone} equals zero.
Since each line-cone can be represented as disjoint union of these degenerated line-cones it is natural to define the Fourier--Laplace transform of line-cones as zero.
Finally, using the Brianchon--Gram identity it is possible to extend the definition of the Fourier--Laplace transform to all elements of the algebra of polyhedra $\mathcal{P}^d$ (for full details see Theorem~8.4 in \cite{barvinok2008integer}).

It is clear from the above discussion that the Fourier--Laplace transform of $f\in \mathcal{P}^d$ can be written in the form
\begin{equation}
	\label{eq:f-l representation}
	\sum_{\mathbf{v}\in \mathcal{V}} e^{\langle \mathbf{v}, \mathbf{z} \rangle} F_{\mathbf{v}}(\mathbf{z}),
\end{equation}
where $F_{\mathbf{v}}$ is a finite linear combination of functions of the form $ \prod_1^d \langle \mathbf{w}_j, \mathbf{z} \rangle)^{-1}$, and $\mathcal{V} \subset \mathbf{R}^d$ is a finite set. For instance, for  $\mathcal{V}$ we can choose the set of geometric vertices of a basic decomposition of~$f$.
Theorem~\ref{thm:sum of line-cones and simplicial cones}' shows that $\mathcal{V}=\mathcal{V}_f$ is an appropriate choice.
We will show later that, in fact,~$\mathcal{V}$ essentially coincides with $\mathcal{V}_f$.

We also should note that $\widehat{f}(\mathbf{z})$ is an analytic function of $\mathbf z$, and so it follows that if 
it vanishes on an open $d$-dimensional set, then it vanishes on all of  $\mathbf{R}^d$. 

\bigskip

\section{Polyhedra whose Fourier--Laplace transform vanishes}
\label{sec:proof of the cone-like sum}

Here we prove Theorem~\ref{thm:f-zero cone is sum of line-cones} for any function from the algebra of polyhedra.

\begin{reptheorem}{thm:f-zero cone is sum of line-cones}
	If $f\in \mathcal{P}^d$ has zero Fourier--Laplace transform, then it has no algebraic vertices, or equivalently by Theorem~\ref{thm:sum of line-cones and simplicial cones}', we have
	\[f= \sum_{i=1}^k \alpha_i [D_i]\]
	for some real numbers $\alpha_i$ and line-cones $D_i$, $i=1, \dots, k$.
\end{reptheorem}


Let $\mathcal{V}_f$ be the set of algebraic vertices of $f$, and assume that $\mathcal{V}_f\ne \emptyset$.
Rearrange the members in the representation of $f$ as the sum from Theorem~\ref{thm:sum of line-cones and simplicial cones}':
\[
f   = \sum_{i=1}^k \alpha_i [D_i]+ \sum_{\mathbf{v} \in \mathcal{V}_f} \tcone(f,\mathbf{v}).
\]
The proof consist of two parts. Lemma~\ref{lem:polyconical zero-f is a zum of line-cones} below implies that the Fourier-Laplace transform of $\tcone(f,\mathbf{v})$ is non-zero because it cannot be represented as a linear combination of line-cones.

All the algebraic vertices $\mathbf{v} \in \mathcal{V}_f$ are  distinct. Then Lemma~\ref{lem:sum of polyconical functions} shows that the Fourier--Laplace transform of their sum should be non-zero as well.

The following lemma is crucial to a lot of the analysis that ensues, showing in particular that if we consider the union of pointed cones 
which lie in the interior of a half space, then the transform of this union cannot vanish.

\begin{lemma}
	\label{lem:upper cones}
	Suppose we have $n$ internally disjoint pointed cones $C_i$ lying in the upper halfspace of $\mathbf R^d$. Let $g=\sum_{i=1}^n \alpha_i [C_i]$.
	Then
	\[
	\widehat g(\mathbf{\cdot})=0 \text{ if and only if all } \alpha_i=0.
	\]
\end{lemma}
\begin{proof}
	Let $C$ be the cone generated by all generators of the cones $C_i$. Thus $C$ is a pointed cone.
	Suppose $\widehat g(\mathbf{z})=0$ for all $\mathbf{z}$ belonging to the interior of the dual cone $C^\circ$.
	We assume by contradiction that all $\alpha_i$ are different from zero.
	Let $\mathbf{w}$ be one of its extreme ray directions of~$C$, and let $C_1$, \dots, $C_k$ be the cones that have $\mathbf{w}$ as a generator.
	We assume without loss of generality that $\|\mathbf{w}\|=1$.
	We are going to show that $\alpha_i=0$ for $i=1,\ldots,k$.
	This contradiction will finish the proof.
		

	\medskip
	From here we proceed by induction on the dimension $d$. For $d=1$ there is a unique cone, which is the ray generated by $\mathbf{w}$.
	So $\widehat g(\mathbf{z})=\frac{\alpha_1}{\langle \mathbf{w}, \mathbf{z} \rangle}$, and it is clear that $\alpha_1=0$.
	The case $d=2$ is also somewhat special since then there is only one cone $C_1$ containing $\mathbf{w}$.
	Let $\mathbf{z}$ be the unique unit vector orthogonal to $\mathbf{w}$ and having $\langle \mathbf{z},\mathbf{w}_j \rangle >0$ for all other generators $\mathbf{w}_j$.
	Define for $\varepsilon >0$
	\[
	\mathbf{u}=\mathbf{u}_{\varepsilon}=\frac {\mathbf{z}+\varepsilon \mathbf{w}}{||\mathbf{z}+\varepsilon \mathbf{w}||}.
	\]
	If $\varepsilon$ is small enough then $\mathbf{u}$ is a unit vector with $\langle \mathbf{u},\mathbf{w}_j \rangle >0$ for all generators including $\mathbf{w}$ as well. It is easy to check using~\eqref{eq:transform of a simplicial cone} that $\widehat {[C_i]}(\mathbf{u})$ are bounded for all $i>1$. Denoting the second generator of $C_1$ by $\mathbf{w}_1$,
 \[
 \widehat {[C_1]}(\mathbf{u})=\frac{| \mathbf{w} \wedge \mathbf{w}_1|}{
	\langle \mathbf{w}, \mathbf{u} \rangle \langle \mathbf{w}_1,\mathbf{u} \rangle}=\frac{| \mathbf{w} \wedge \mathbf{w}_1|}
	{\varepsilon \langle \mathbf{w}_1, \mathbf{z}\rangle(1+o(1))},
 \]
So $\widehat g(\mathbf{u})=0$ for all $\varepsilon >0$ is only possible if $\alpha_1=0$.
	
	Now for the step $d-1 \to d$ with $d>2$, we write $G=\{\mathbf{x} \in \mathbf R^d: \langle \mathbf{w},\mathbf{x}\rangle=0\}$; this is a copy of $\mathbf{R}^{d-1}$ of course.
	Let $C_i'$ denote the orthogonal projection of $C_i$ to $G$ for $i=1,\ldots,k$.
	Let $\mathbf{z} \in G$ be a unit vector with $\langle \mathbf{z} ,\mathbf{w}_j \rangle>0$ for all generators different from $\mathbf{w}$.
	The set of such $\mathbf{z}$'s is exactly the set of unit normal vectors to those support hyperplanes of $C$ that contain only $\mathbf{w}$ from the generators of $C$.  
	Observe now that all $C_i'$, $i=1,\ldots,k$ lie in $\{\mathbf{x} \in G: \langle \mathbf{x}, \mathbf{z}\rangle>0\}$ which is an open halfspace in $G$.
	Define $\mathbf{u}=\mathbf{u}_{\varepsilon}$ the same way as above; so $\mathbf{u}$ is a unit vector with $\langle \mathbf{w}_j, \mathbf{u} \rangle >0$ for all generators including $\mathbf{w}$ as well.

	Note that for $i=1,\ldots,k$
	\[
	\widehat {[C_i]}(\mathbf{u})=\frac{\widehat {[C'_i]}(\mathbf{z})+O(\varepsilon)}{\varepsilon}.
	\]	
	Indeed, suppose $C_i$ is generated by vectors $\mathbf{w}_1$, \dots $\mathbf{w}_{d-1}$, $\mathbf{w}_d=\mathbf{w}$.
	Denote by $\mathbf{w}'_j$ the projection of~$\mathbf{w}_j$ on the hyperplane $G$.
	Then
	\[
	\widehat {[C_i]}(\mathbf{u})=\frac{| \mathbf{w}_1\wedge \dots \wedge \mathbf{w}_{d}|}{\prod_{j=1}^d
	\langle \mathbf{w}_j, \mathbf{u} \rangle}=
	\frac{|\mathbf{w}'_1\wedge \dots \wedge \mathbf{w}'_{d-1}|}{\prod_{j=1}^{d-1}(\langle
	\mathbf{w}'_j, \mathbf{z} \rangle+ \varepsilon \langle \mathbf{w}, \mathbf{w}_j\rangle)} \cdot
	\frac{\|\mathbf{w}\|}{\varepsilon \|\mathbf{w}\|^2}
	=(\widehat {[C_i']}(\mathbf{z})+O(\varepsilon))\cdot \frac{1}{\varepsilon}.
	\]
	Note that $\widehat {[C_i]}(\mathbf{u})$ for $i>k$ is bounded as $\varepsilon$ goes to zero. Therefore
	\[
	0=\widehat g(\mathbf{u})=\frac{1}{\varepsilon} \sum_1^k \alpha_i \widehat {[C_i']}(\mathbf{z}) +O(1).
	\]
	Multiplying by $\varepsilon$ and taking the limit $\varepsilon \to 0$ gives the equation $\sum_1^k \alpha_i \widehat {[C_i']}(\mathbf{z})=0$.
	This holds for all unit vectors $\mathbf{z}\in G$ for which $\langle \mathbf{z}, \mathbf{w}_j \rangle>0$ for every generator of the cones $C_i$, $i=1,\ldots,k$, which is the same as $\langle \mathbf{z}, \mathbf{w}_j' \rangle>0$ for every generator of the cones $C'_i$, $i=1,\ldots,k$.
	Observe that $\sum_1^k \alpha_i \widehat {[C_i']}(\mathbf{z})$ is the Fourier--Laplace transform (taken in $G=\mathbf{R}^{d-1}$) of the function $\sum_1^k \alpha_i [C'_i]$ and so by the induction hypothesis all $\alpha_i=0$ for $i=1,\ldots,k$.
\end{proof}

\medskip

\begin{lemma}
	\label{lem:polyconical zero-f is a zum of line-cones}
	A polyconical function $g$ has zero Fourier--Laplace transform if and only if it is a linear combination of line-cones.
\end{lemma}
\begin{proof}
One direction is easy: if $g$ is a linear combination of line-cones, then it has zero Fourier--Laplace transform because all line-cones have zero Fourier--Laplace transform.
For the other direction we consider the decomposition of $g$ into simplicial cones:
\[
	g=\sum_{i=1}\alpha_i[C_i].
\]

We will show how to reduce the proof to the context of Lemma \ref{lem:upper cones}, namely,	
it suffices to prove the theorem in a simpler case,  when

(i) all the $C_i$ are internally disjoint pointed cones, and

(ii) there is a unit vector $\mathbf{v}$ such that $\langle \mathbf{v},\mathbf{w} \rangle >0$ for every generator $\mathbf{w}$ of every $C_i$.

To see this we start with a simplicial cone $C$ whose generators are $\mathbf{w}_1,\ldots,\mathbf{w}_d$, that is, these vectors are linearly independent and every point in $C$ is a linear combination with non-negative coefficients of the $\mathbf{w}_j$, that is, $C=\pos\{\mathbf{w}_1,\ldots,\mathbf{w}_d\}$. Let, further, $\mathbf{v}$ be a unit vector with $\langle \mathbf{v},\mathbf{w}_j \rangle \ne 0$ for any $j$. Define $\mathbf{w}_j^*$ as $\mathbf{w}_j$ if $\langle \mathbf{v},\mathbf{w}_j \rangle >0$ and $-\mathbf{w}_j$ if $\langle \mathbf{v},\mathbf{w}_j \rangle <0$. Set $C^*=\pos\{\mathbf{w}_1^*,\ldots,\mathbf{w}_d^*\}$. We start with a simple

\begin{claim}
	\label{cl:all cones in upper halfspace}
	There are coefficients $\varepsilon \in \{-1,1\}$ and $\varepsilon_i\in \{-1,0,1\}$ and line-cones $D_i$ such that $[C]=\varepsilon[C^*]+\sum_i^d \varepsilon_i[D_i]$.
\end{claim}

Proof of the claim.  \ \ 
	There is nothing to prove if $\langle \mathbf{v},\mathbf{w}_j \rangle >0$ for all $j$. Assume $\langle \mathbf{v},\mathbf{w}_1 \rangle <0$, say, and observe that
	\begin{eqnarray*}
	[C]&=&[\pos\{\mathbf{w}_1,\ldots,\mathbf{w}_d\}]\\
	 &=& [\pos\{\mathbf{w}_1,-\mathbf{w}_1,\mathbf{w}_2,\ldots,\mathbf{w}_d\}]\setminus [\pos\{-\mathbf{w}_1,\mathbf{w}_2,\ldots,\mathbf{w}_d\}].
	\end{eqnarray*}
	The cone $\pos\{\mathbf{w}_1,-\mathbf{w}_1,\mathbf{w}_2,\ldots,\mathbf{w}_d\}$ is a line-cone and the simplicial cone
	 $\pos\{-\mathbf{w}_1,\mathbf{w}_2,\ldots,\mathbf{w}_d\}=\pos\{\mathbf{w}_1^*,\mathbf{w}_2,\ldots,\mathbf{w}_d\}$ has one fewer generators than $C$ with $\langle \mathbf{v}, \mathbf{w} \rangle <0$. Continuing this way we end up with the required formula, finishing the proof of the claim.

 \medskip
 \noindent 
We return now to the reduction of our Lemma to Lemma  \ref{lem:upper cones}.  
    Applying the Claim to each cone $C_i$ we have
	\[
	g=\sum_{i=1}^n \alpha_i [C_i]=\sum_{i=1}^n \alpha_i\varepsilon_i[C_i^*]
	+\sum_{j=1}^k\gamma_j[D_j]
	\]
	where $\varepsilon_i=\pm 1$ and the $D_j$ are line-cones. It follows that the Fourier--Laplace transform of
	$g^*=\sum_{i=1}^n \alpha_i\varepsilon_i[C_i^*]$ is zero.

	\medskip
	The advantage of $g^*$ is that it is a linear combination (of the indicator functions) of simplicial cones $C^*_i$ and each $C^*_i$ lies in the open halfspace $H=\{\mathbf{x}\in \mathbf R^d: \langle \mathbf{x}, \mathbf{v}\rangle>0\}$.
	
	Consider now all the hyperplanes containing a facet of some $C_i^*$. The complement of their union in $\mathbf R^d$ consists of finitely many cones, and on each such cone, $F$ say, $g^*$ is constant. This constant is zero if the cone $F$ does not lie completely in $H$. Thus $g^*=\sum_{j=1}^m\beta_j[F_j]$ where each $\beta_j\ne 0$ and the $F_j$ are pairwise internally disjoint cones, each contained in $H$. This finishes the proof of the reduction, and of the Lemma.

\end{proof}

\begin{lemma}
	\label{lem:sum of polyconical functions}
	Suppose we are given a collection of $n$ distinct points $\textbf{v}_i \in \mathbf{R}^d$
	and $n$ polyconical functions $f_i$, whose Fourier--Laplace transform $\widehat {f_i}$ is not the zero function.
	
	Let $g (\mathbf{z}) := \sum_{i=1}^n \alpha_i f_i(\mathbf{z}-\mathbf{v}_i)$, with all $\alpha_i \in \textbf{R}$.
	Then
	\[
	\widehat g(\mathbf{\cdot}) = 0\text{ if and only if all }\alpha_i = 0.
	\]
\end{lemma}

\begin{proof}
	Let $P$ be the convex hull of the points $\mathbf{v}_i$.
	Without loss of generality we can assume that~$\mathbf{v}_1$ is an extreme point of $P$.
	We translate everything so that $\mathbf{v}_1$ becomes the origin.
	Let $K$ be the dual cone to the tangent cone of $P$ at $\mathbf{v}_1$, and fix $\mathbf{u}$ to be any point in the interior of $K$.
	
	Denote by $g_i(\mathbf{z}) := f_i(\mathbf{z}-(\mathbf{v}_i-\mathbf{v}_1))$.
	It is enough to prove that $\sum \alpha_i g_i$ has zero Fourier--Laplace transform if and only if all $\alpha_i=0$.

As the apex of every cone defining $g_i$ is $\mathbf{v}_i-\mathbf{v}_1$, and its Fourier--Laplace transform is of the form
\[
e^{\langle \mathbf{v}_i-\mathbf{v}_1, \mathbf{z}\rangle} F_i(\mathbf{z}),
\]
where $F_i(\mathbf{z})$ is a linear combination of functions of the form 
$(\prod_1^d \langle\mathbf{w}_j,\mathbf{z}\rangle)^{-1}$.
Then for every $i > 1$ the value $\widehat{g_i}(t\textbf{u})$ tends to $0$ like $o(e^{tc_i})$, with $c_i<0$, as $t \rightarrow \infty$. But the function $\widehat {g_1}(\mathbf{z})$ does not have an exponential in the numerator, since its associated vertex is the origin, and hence  $\widehat {g_1}(t\textbf{u}) = t^{-d} \cdot \widehat {g_1}(\textbf{u})$ tends to $0$ like $O(t^{-d})$. Therefore $\widehat {g_1}(\mathbf{u}) = 0$ for all $\mathbf{u} \in K$.	Since $K$ is $d$-dimensional, we see that $\widehat {g_1}(\textbf{z}) = 0$ for all $\mathbf{z}$ and hence $\widehat {f_1}$ vanishes identically.
	This last conclusion contradicts our assumption about the $f_i$'s.
\end{proof}

\begin{remark}
	\label{rem:unique representation}
	Using the same argument in the Lemma above, we return to formula~\eqref{eq:f-l representation} and claim~$F_{\mathbf{v}}$ is zero if $\mathbf{v}\notin \mathcal{V}_f$.
	This will show that  the representation of the Fourier-Laplace transform in the form of \eqref{eq:f-l representation} is essentially unique.
	The proof is simple. Suppose there are two representations.  We then subtract the two representations from each other to obtain  
	\[
		0=\sum_{\mathbf{v}\in \mathcal{V}'} e^{\langle \mathbf{v}, \mathbf{z} \rangle} F_{\mathbf{v}}(\mathbf{z}),
	\]
	for some finite set $\mathcal{V}'$.
	Now repeating the arguments from the proof of Lemma~\ref{lem:sum of polyconical functions} leads to a contradiction.
\end{remark}

\section{Signed sections of polyhedra}
\label{sec:signed section}
\medskip

The following theorem shows the strong correlation between $f\in \mathcal{P}^d$ and its signed sections~$f_h$.
\begin{theorem}
	\label{thm:section theorem}
	Let $f \in \mathcal{P}^d$. Then $\widehat {f}(\cdot)=0$ if and only if $\widehat {f}_h=0$ for any oriented hyperplane~$h$.
\end{theorem}

We first prove it for polyconical functions.

\begin{lemma}
	\label{lem:section of polyconical functions}
	The Fourier--Laplace transform of a polyconical function $f$ is zero if and only if for all hyperplanes $h$ through the origin, the $(d-1)$-dimensional Fourier--Laplace transform of the signed section~$f_h$ is zero.
\end{lemma}

\begin{proof}
	Note that the $(d-1)$-dimensional Fourier--Laplace transform of every signed section of a line-cone is zero, as the signed section itself is a line cone.
	
	Now applying Claim~\ref{cl:all cones in upper halfspace} to the cones generating $f$ we get the function~$f^*$ whose support lies in a halfspace $H$.
	From Lemma~\ref{lem:upper cones} it follows that $\widehat f(\cdot)$ is zero if and only if $f^*$ is zero.
	
	But $f^*$ vanishes if and only if all of its signed sections vanish.
	Supports of signed sections of~$f^*$ lie in $H$.
	Therefore, by Lemma~\ref{lem:upper cones} again, the condition that the Fourier--Laplace transform of all signed sections of $f^*$ are zero is equivalent to the condition that all their $(d-1)$-dimensional Fourier--Laplace transforms are zero.
	But the Fourier--Laplace transform of a signed section $f^*_h$ coincides with the Fourier--Laplace transform of a signed section $f_h$, for any $h$, since they differ from each other on linear combinations of line-cones.
\end{proof}

\begin{proof}[Proof of Theorem~\ref{thm:section theorem}]
	By Theorem~\ref{thm:f-zero cone is sum of line-cones}, if $\widehat {f}(\cdot)=0$, then at each point $\mathbf{v}$,  the Fourier--Laplace transform of the tangent cone to $f$ equals zero, and vice versa.
	
	Suppose $\widehat {f}(\cdot)=0$ and $h$ is a some hyperplane.
	Consider the tangent cone $C=\tcone(f, \mathbf{v})$ of a point $\mathbf{v}$ from $h$. By Lemma~\ref{lem:section of polyconical functions}  its Fourier--Laplace transform is zero, therefore the Fourier--Laplace transform of its signed section $C_h$ is also zero. This holds for any point~$\mathbf{v}$ from $h$ and  therefore $\widehat{f}_h=0$.
	
	In the other direction the proof is analogous.  We let $\mathbf{v}$ be any point, and $C:=\tcone(f, \mathbf{v})$. For any hyperplane $h$ through $\mathbf{v}$ we have $\widehat{f}_h=0$, and therefore by Lemma~\ref{lem:sum of polyconical functions}, $\widehat C_h=0$.
	Lemma~\ref{lem:section of polyconical functions} now implies that $\widehat C =0$.
	Since this holds for any point $\mathbf{v}$, we conclude that $\widehat {f}(\cdot)=0$.
\end{proof}

\bigskip

\section{What is a vertex of a generalized polytope?}
\label{sec:definition discussion}

There are several ways to define a vertex of generalized polytope $P$ (or of $f\in \mathcal{P}^d$).
In this section we discuss these possibilities.

In the introduction $\mathbf{v}$ is defined as a geometric vertex of $P$ if $\tcone(P,\mathbf{v})$ is not a line-cone.
Figure~\ref{fig:tangent edges}  shows a generalized polytope in which two edges touch each other at the point 
$\mathbf{v}$.
According to the above definition, $\mathbf{v}$ is a geometric vertex.
But it is not very natural to consider this point as a vertex, because it appears occasionaly and can disappear after a slight mutation of polytope.

One may try to define a geometric vertex when not all connected components of $\tcone(P,\mathbf{v})$ are line-cones.
Such a vertex is shown on Figure~\ref{fig:glued faces}.
Is it natural to take this point for a vertex? Perhaps not.

%
%
%

\begin{center}
\parbox{5.5cm}
{\begin{center}
\includegraphics{fig-vertex-1.mps}\\
\vskip 1.6cm
\figure \label{fig:tangent edges}
\end{center}
}
\parbox{5cm}{\begin{center}
\includegraphics{fig-vertex-2.mps}\\
\vskip -0cm
\figure \label{fig:glued faces}
\end{center}
}	
\parbox{6cm}{\begin{center}
\includegraphics{fig-vertex-3.mps}\\
\vskip 0.9cm
\figure \label{fig:triple vertex}
\end{center}
}	
\end{center}
\bigskip


Here is a another way to define a vertex (see~\cite{gravin2012moments}).

\begin{definition}
	\label{def:temporary definition through cones}
	A point $\mathbf{v}$ is a \textbf{combinatorial vertex} of a generalized polytope $P$ if its tangent cone cannot be partitioned into line-cones.
\end{definition}

Lemma 5 in \cite{gravin2012moments} shows that this definition is equivalent to the following one.

\begin{definition}
	\label{def:temporary definition through triangulations}
	 A point $\mathbf{v}$ is a \textbf{combinatorial vertex} of a generalized polytope $P$ if, in every dissection $P$ into disjoint simplices, the point $\mathbf{v}$ is a vertex of some simplex of the dissection.
\end{definition}


%

It is well-known that some non-convex polytopes $P$ cannot be triangulated into simplices using only vertices of $P$ .
The most famous example is the Sch\"onhardt polytope \cite{schonhardt1928zerlegung}, which is a modification of the triangular prism.

So if we try to define a triangulation of non-convex polytopes we are either forced to use simplices with vertices not only from the set of vertices of $P$ (which are clearly discernible in the case of the Sch\"onhardt polytope), or we are forced to use overlapping simplices by subtracting simplices from each other (see \cite[Remark 5]{gravin2012moments} for a representation of the Sch\"onhardt polytope).
In view of this phenomena, it is natural to modify the latter two definitions by allowing signed versions instead of unions.
We thus arrive at our definition of algebraic vertices, which we repeat from the introduction.


\begin{repdefinition}{def:indicator definition of Fourier vertex through cones}
	For a generalized polyhedron $P$, a point $\mathbf{v} \in P$ is called an {\bf algebraic vertex} of $P$ if the indicator function of its tangent cone $\tcone(P, \mathbf{v})$ cannot be represented as a linear combination of indicator functions of line-cones.
\end{repdefinition}

\begin{definition}
	\label{def:vertex pm-definition}
	A point $\mathbf{v}$ is an \textbf{algebraic vertex} of $P$ if in every representation of the indicator function of $P$ as a linear combination of indicator functions of simplices, the point $\mathbf{v}$ is a vertex of some simplex.
\end{definition}

Theorem \ref{thm:sum of simplices} shows that the set of vertices that satisfy these two definitions  coincide.

\begin{remark}
	\label{rem:fanappie vertices}
	
	In~\cite{gravin2012moments}, instead of the Fourier--Laplace transform the authors consider the Fantappi\`e transform $\mathcal{F}_P$ (which is closely related to the Fourier--Laplace transform) of a generalized polytope $P$, see
	the discussion following Theorem~\ref{thm:sum of simplices}.
	The authors of \cite{gravin2012moments} show that if the set of combinatorial vertices of $P$ is in general position (i.e. the affine span of each $(d+1)$-tuple equals $\mathbf{R}^d$) then it is a set of vertices $V_{\mathcal{F}}$ appearing in $\mathcal{F}_P$.
	Using our approach it is possible to show that $V_{\mathcal{F}}$ is exactly the set of algebraic vertices of $P$.
	The proof is similar to the proof of Lemma~\ref{lem:sum of polyconical functions} and Remark~\ref{rem:unique representation}.
	
\end{remark}

Finally, Figure~\ref{fig:triple vertex} shows the tangent cone of the generalized polytope $P$ which is defined as a union of three cones having a common apex $\mathbf{v}$, such that their sides lie on three lines.
The Fourier--Laplace transform of $\tcone(P,\mathbf{v})$ is zero, so $\mathbf{v}$ is not a vertex (see \cite{barvinok2008integer}, Problem~9.3 for a more general statement).  But is it not, really? The interested reader should decide.



%

\subsection*{Acknowledgement}
We would like to thank Dmitrii Pasechnik and Boris Shapiro for drawing our attention to these questions and for fruitful discussions.
We are also very grateful to the anonymous reviewer of the journal Advances in Mathematics for many useful remarks that greatly improved the text.

The first author was supported by People Programme (Marie Curie Actions) of the European Union's Seventh Framework Programme (FP7/2007-2013) under REA grant agreement n$^\circ$[291734].
The second author acknowledges support from ERC Advanced Research Grant no 267165 (DISCONV) and by Hungarian National Research Grant K 111827. The third author was supported in part by ICERM, Brown University, and in part by 
the FAPESP grant Proc. 2103/03447-6, Brasil. 

A part of the paper was completed while the authors were supported by
Singapore MOE Tier 2 Grant MOE2011-T2-1-090 (ARC 19/11) and the Institute for
Mathematical Sciences at the National University of Singapore, under the program
``Inverse Moment Problems: the Crossroads of Analysis, Algebra, Discrete
Geometry and Combinatorics''.

\nocite{brion1988points,pukhlikov1992riemann,barvinok2008integer,frettloh2009lonely,gravin2012moments, beck2007computing}

\bibliographystyle{abbrv}
\bibliography{vertex.bib}{}
	
\end{document}